\documentclass[12pt]{amsart}
\usepackage[top=1.26in, bottom=1.26in, left=1.5in, right=1.5in]{geometry}

\usepackage{amsmath}
\usepackage{enumerate}
\usepackage{float}
\usepackage{comment}
\usepackage{fancyvrb}
\usepackage{caption}
\usepackage{tikz-cd}
\usepackage{subcaption}

\usepackage{mathtools}
\usepackage{scalerel}
\usepackage{amssymb}
\usepackage{hyperref}
\hypersetup{
    colorlinks=true,
    linkcolor=blue,
    filecolor=magenta,      
    urlcolor=blue,
    citecolor=blue,
}
\usepackage{color}
\usepackage{graphicx}
\usepackage{mathrsfs}

\usepackage{mathabx}
\usepackage{amsthm}
\usepackage{color}

\usepackage[font=small]{caption}
\usetikzlibrary{cd}
\usetikzlibrary{decorations.markings}

\theoremstyle{definition}
\newtheorem{thm}{Theorem}[section]
\newtheorem{prop}[thm]{Proposition}

\newtheorem{lemma}[thm]{Lemma}
\newtheorem{defn}[thm]{Definition}

\newtheorem{ex}[thm]{Example}

\def\S{\mathfrak{S}}

\def\k{\mathbf{k}}

\def\Par{\operatorname{Par}}

\def\SSYT{\operatorname{SSYT}}
\def\SYT{\operatorname{SYT}}

\def\BPD{\operatorname{BPD}}

\def\id{\operatorname{id}}
\def\ins{S}

\def\colread{\operatorname{colread}}

\def\shape{\operatorname{sh}}

\def\perm{\operatorname{perm}}

\def\len{\operatorname{length}}
\def\maxcode{\operatorname{maxcode}}
\def\code{\operatorname{code}}
\def\words{\operatorname{words}}

\def\+{\includegraphics[scale=0.4]{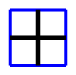}}

\def\bt{\includegraphics[scale=0.4]{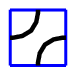}}
\def\rt{\includegraphics[scale=0.4]{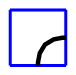}}
\def\jt{\includegraphics[scale=0.4]{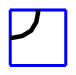}}

\def\mindroop{\textsf{min-droop}}

\title{Knuth moves for Schubert polynomials}

\author{Daoji Huang}
\thanks{DH was supported by NSF-DMS2202900.}
\author{Pavlo Pylyavskyy}

\date{}

\begin{document}

\maketitle

\begin{abstract}
In our previous work we have introduced an analogue of Robinson-Schensted-Knuth correspondence for Schubert calculus of the complete flag varieties. The objects inserted are certain biwords, the outcomes of insertion are bumpless pipe dreams, and the recording objects are decorated chains in Bruhat order. In this paper we study a class of biwords that have a certain associativity property; we call them plactic biwords. We introduce analogues of Knuth moves on plactic biwords, and prove that any two plactic biwords with the same insertion bumpless pipe dream are connected by those moves. 
\end{abstract}

\section{Introduction}
The classical Robinson-Schensted-Knuth (RSK) correspondence \cite{robinson1938representations, schensted1961longest, knuth1970permutations}, in its most general form, gives a bijection between matrices $A$ with nonnegative integers and paris of semistandard Young tableaux $(P,Q)$ of the same shape $\lambda$, where the size of $\lambda$ is equal to the sum of entries in $A$. This correspondence has numerous consequences and applications; for example it gives a bijective proof of the celebrated Cauchy's identity of symmetric functions. There are also numerous specializations and variants of the RSK correspondence that are useful in different contexts. In particular, one version gives a bijection between the set of finite-length words on the alphabet $\{1,2,3,\cdots\}$ and pairs of Young tableaux $(P,Q)$ where $P$ is semistandard and $Q$ is standard with the same shape. Here $P$ is known as the \emph{insertion tableaux} and $Q$ as the \emph{recording tableaux}.  This gives a combinatorial interpretation of the  identity \[(x_1+x_2+\cdots)^m=\sum_{\lambda \vdash m}|\SYT(\lambda)|s_\lambda(x_1,x_2,\cdots).\] A central property of this RSK correspondence is the following: two words $\mathbf{a}_1$ and $\mathbf{a}_2$ have the same insertion tableaux if and only if they are related by the \emph{elementary Knuth relations}:
\[bac\sim bca \text{ if }a<b\le c\]
\[acb\sim cab \text{ if }a\le b< c.\]
This led to the definition of \emph{plactic monoid} by Lascoux and Schützenberger \cite{MR646486}, see also \cite[Chapter 5]{MR1905123}, which is a monoid whose generators are elements in a totally ordered alphabet and relations are the elementary Knuth relations. 

In our previous work \cite{HP}, we generalized Schensted's left and right insertion algorithms on semistandard Young tableaux, which give a combinatorial model of Schur polynomials (or functions), to bumpless pipe dreams of Lam, Lee, and Shimozono \cite{LLS}, which give a combinatorial model for Schubert polynomials \cite{MR660739}. Based on these algorithms we introduced two versions of the RSK correspondence, where we generalize the input from words to certain biwords, the insertion objects from SSYTs to BPDs, and recording objects from SYTs to certain decorated chains in Bruhat order.  This generalized RSK correspondence interprets the polynomial identity 
\[\S_{k_1}\S_{k_2}\cdots\S_{k_m}=\sum_{\pi\in S_\infty :\ell(\pi)=m}|chs_{(k_1,\cdots,k_m)}(\pi)|\S_\pi.\]
where $k_1,\cdots,k_m$ are positive integers, $\S_k=x_1+\cdots+x_k$ for any $k$, 
$\S_\pi$ is the Schubert polynomial for $\pi$, and 
\[chs_{(k_1,\cdots,k_m)}(\pi)=\{(\id\lessdot_{k_1} \pi_1\lessdot_{k_2}\cdots\lessdot_{k_{\ell(\pi)}}\pi_{\ell(\pi)}=\pi)\},\]
namely the set of complete  chains in mixed $k$-Bruhat order where the decoration is $(k_1,\cdots,k_m).$ Notice that it's possible that $chs(\pi)$ is empty. 

In classical RSK, Schensted's row (right) and column (left) insertions always commute. Therefore, given a word, we may read it from left to right and iteratively perform right insertion, read it from right to left and iteratively perform left insertion, or more generally start with a letter in the middle and iteratively construct a consecutive subword by adding one more letter, and whenever we read a letter to the left we perform left insertion and whenever we read a letter to the right we perform right insertion. The insertion object (SSYT) is always well-defined regardless of these various choices of insertion order. However, this is not true in general for our generalized RSK insertion. This poses an obstacle 
to generalize the classical theory of plactic monoid to the Schubert setting. In this paper, we identify a simple condition on biwords (Definition~\ref{def:assoc-biwords}) under which the insertion BPD is well-defined regardless of order of insertion (Proposition~\ref{prop:associativity}) and generalize classical Knuth relations to this restricted set of biwords, which we will call \emph{plactic biwords} (Definition~\ref{def:knuth}). We furthermore study the basic properties of these generalized Knuth relations and plactic biwords. 
\begin{defn}
\label{def:assoc-biwords}
A biletter is a pair of positive integers $\binom{a}{k}$ where $a\le k$.
A \textbf{plactic biword} is a word of biletters $\binom{\mathbf{a}}{\mathbf{k}}=\binom{a_1,\cdots, a_\ell}{k_1,\cdots, k_\ell}$, where $k_i\ge k_{i+1}$ for each $i$. If furthermore $k_1=\cdots =k_\ell =k$ for a fixed $k$, we call the plactic biword a \textbf{$k$-biword}. We may write a $k$-biword $\binom{\mathbf{a}}{\mathbf{k}}$ as $\binom{\mathbf{a}}{k}$.
\end{defn}

The following statement shows that the insertion object of a plactic biword is well-defined regardless of choices of order of insertion. Recall that a given biword $Q=\binom{a_1,\cdots,a_\ell}{k_1,\cdots,k_\ell}$, \cite{HP} defines a map $\mathcal{L}(Q)=(\varphi_L(Q), ch_L(Q))$ where $\varphi_L(Q)$ is the BPD obtained by reading $Q$ from right to left and successively performing left insertion, and 
$ch_L(Q)$ is the recording chain in mixed $k$-Bruhat order with edge labels $k_\ell,\cdots, k_1$,
as well as a map $\mathcal{R}(Q)=(\varphi_R(Q), ch_R(Q))$ where $\varphi_R(Q)$ is the BPD obtained by reading $Q$ from left to right and successively performing right insertion, and $ch_R(Q)$ is the recording chain in mixed $k$-Bruhat order with edge labels $k_1,\cdots, k_\ell$.
For details of these insertion algorithms see \cite[Section 3]{HP}.  
\begin{prop}
\label{prop:associativity}
Suppose $Q=\binom{\mathbf{a}}{\mathbf{k}}$ is a plactic biword of length $\ell$.
Let $Q_1, Q_2,\cdots,Q_\ell$ be a sequence of biwords such that  $\len(Q_i)=i$, $Q_i$ is a consecutive subword of $Q$, and for each $1\le i < \ell$ $Q_{i+1}=\binom{a}{k} Q_i$, or $Q_{i+1}=Q_i\binom{a}{k}$ for some $a$ and $k$.   Let $D_1:=\varphi_L(Q_1)=\varphi_R(Q_1)$, and for each $1\le i < \ell$, define
\[D_{i+1}=\begin{cases}
\binom{a}{k}\rightarrow D_i & \text{ if } Q_{i+1}=\binom{a}{k}Q_i\\
D_i\leftarrow \binom{a}{k} & \text{ if } Q_{i+1}=Q_i\binom{a}{k}.
\end{cases}\]
Then $D_\ell$ is independent of the choice of the sequence of biwords. We write $\varphi(Q)$ for $D_\ell$.
\end{prop}
\begin{proof}
This follows from \cite[Theorem 3.3]{HP} by induction.
\end{proof}
\begin{defn}
\label{def:knuth}
We define the \textbf{generalized Knuth relations} on plactic biwords as follows:
\begin{enumerate}[(1)]
    \item $\left(\begin{smallmatrix}\cdots & b & a & c & \cdots \\ \cdots & k & k & k & \cdots\end{smallmatrix}\right)\sim \left(\begin{smallmatrix}\cdots & b & c & a & \cdots \\ \cdots & k & k & k & \cdots\end{smallmatrix}\right)$ if $a<b\le c$
    \item $\left(\begin{smallmatrix}\cdots & a & c & b & \cdots \\ \cdots & k & k & k & \cdots\end{smallmatrix}\right)\sim \left(\begin{smallmatrix}\cdots & c & a & b & \cdots \\ \cdots & k & k & k & \cdots\end{smallmatrix}\right)$ if $a\le b< c$
    \item $\left(\begin{smallmatrix}\cdots & a & b & \cdots \\ \cdots & k & k & \cdots\end{smallmatrix}\right)\sim \left(\begin{smallmatrix}\cdots & a & b & \cdots \\ \cdots & k+1 & k & \cdots\end{smallmatrix}\right)$ if $a\le b$
    \item $\left(\begin{smallmatrix}\cdots & b & a & \cdots \\ \cdots & k+1 & k+1 & \cdots\end{smallmatrix}\right)\sim \left(\begin{smallmatrix}\cdots & b & a & \cdots \\ \cdots & k+1 & k & \cdots\end{smallmatrix}\right)$ if $a< b$.
\end{enumerate}

\end{defn}
Note that these relations are defined on the set of plactic biwords only; namely we do not apply the relation (3) or (4) if the resulting word is no longer plactic. Our main theorem is as follows.

\begin{thm}
\label{thm:main}
For any $D\in\BPD(\pi)$, the set of plactic biwords
\[\words(D):=\{Q: \varphi(Q)=D\}\] is connected by the generalized Knuth relations.
\end{thm}

An immediate consequence of Theorem~\ref{thm:main} is that the set of  plactic biwords that insert into $D$ for any $D\in\BPD(\pi)$ using $\mathcal{L}$ (resp., $\mathcal{R}$) is in bijection with the set of mixed $k$-chains from $\id$ to $\pi$ with weakly increasing (resp., weakly decreasing) labels.

Equipped with the formulation of plactic biwords, we may restate our structure constant rule \cite{HP} for the separated descent case (see \cite{allentalk,huang2021schubert}, and also \cite{kogan2000schubert, lenart2010growth}) as counting certain pairs of plactic biwords instead of growth diagrams. Following \cite{HP}, given a permutation $\pi$, we let $d_1(\pi)$ denote the first descent of $\pi$ and $d_2(\pi)$ the last descent of $\pi$. 
\begin{thm}
    Let $\pi, \rho$, $\sigma$ be permutations such that $d_1(\pi)\ge d_2(\rho)$. Then for any $D\in\BPD(\sigma)$, as well as 
    \[ch_\pi=(\id\lessdot_{k_1}\pi_1\lessdot_{k_2}\cdots\lessdot_{k_{\ell(\pi)}}) \]
    where $d_1(\pi)\le k_1\le \cdots \le k_{\ell(\pi)}$, and 
    \[ch_\rho = (\id\lessdot_{m_1}\rho_1\lessdot_{m_2}\cdots \lessdot_{m_{\ell(\rho)}} \rho)\]
    where $d_2(\rho)\ge m_1\ge\cdots \ge m_{\ell(\rho)}$, 
    \[c_{\pi,\rho}^{\sigma}=\#\{(Q_\pi, Q_\rho): ch_L(Q_\pi)=ch_\pi, ch_R(Q_\rho)=ch_\rho, \varphi(Q_\pi Q_\rho)=D\}.\]
    Notice that by assumptions on descents, $Q_\pi Q_\rho$ is a plactic biword whose second row is  $k_{\ell(\pi)},\cdots, k_1, m_1,\cdots, m_{\ell(\rho)}$. 
\end{thm}

In Section~\ref{sec:properties} we establish basic properties of plactic biwords and in Section~\ref{sec:main-proof} we prove our main theorem. In particular, we assume the knowledge of \cite[Section 3]{HP} for these developments. We remark that in the case of plactic biwords, the left and right insertion algorithms simplify, and the details are stated precisely in \cite[Lemma 5.1 and Lemma 5.2]{HP}.

\section{Properties of plactic biwords}
\label{sec:properties}
\begin{prop}
    If $Q_1$ and $Q_2$ are plactic biwords such that $Q_1\sim Q_2$, then $\varphi(Q_1)=\varphi(Q_2)$.
\end{prop}
\begin{proof}
    It suffices to show when $Q_1$ and $Q_2$ differ by a single generalized Knuth move, and by Proposition~\ref{prop:associativity} it suffices to consider the case when they are length 3 or 2 and all biletters participate in the application of the relation. If $Q_1$ and $Q_2$ differ by relations (1) and (2) in Definition~\ref{def:knuth}, which corresponds to the ordinary Knuth moves, the claim follows from the observation that our insertion algorithms on $k$-biwords correspond to classical Schensted's insertion algorithms. For relation (3), we consider first the insertion of a single biletter $\binom{b}{k}$. The resulting BPD contains a single blank on the diagonal at $(b,b)$, a single $\+$ at $(k+1,k+1)$, and a $\jt$ at each $(i,i)$ for $b<i\le k$. If we then left insert $\binom{a}{k}$ for $a \le b$, the final $\mindroop$ is always at $(k,k+1)$ which causes the pipes that exit from row $k$ and row $k+2$ to cross. Therefore we see that left-inserting $\binom{a}{k+1}$ would have the same effect. For relation (4), we consider first the insertion of a single biletter $\binom{b}{k+1}$. The resulting BPD contains a single blank on the diagonal at $(b,b)$, a single $\+$ at $(k+2,k+2)$, and a $\jt$ at each $(i,i)$ for $b<i\le k+1$. If we then right insert $\binom{a}{k+1}$ for $a<b$, during the algorithm there must be a $\mindroop$ at $(b-1,b-1)$ followed by a $\mindroop$ at $(b,b-1)$, by definition of right insertion. The final $\mindroop$ must be at $(k+1,k)$, causing the pipes that exit from row $k$ and $k+2$ to swap. We can then see that right inserting $\binom{a}{k}$ instead of $\binom{a}{k+1}$ has the same effect.
\end{proof}

\begin{defn}
    Let $\pi$ be a permutation. Define $\code(\pi):\mathbb{Z}^{>0}\to \mathbb{Z}^{\ge 0}$ such that \[\code(\pi)(i)=\#\{j>i:\pi(j)<\pi(i)\}.\] In other words, $\code(\pi)(i)$ counts the number of blank boxes in the $i$th row of the Rothe diagram of $w$. This permutation statistics is known as the \textbf{inversion code} or \textbf{Lehmer code} of $\pi$.
\end{defn}
\begin{defn}
Let $\pi$ be a permutation. Define $\maxcode(\pi):\mathbb{Z}^{>0}\to \mathbb{Z}^{\ge 0}$ such that
\[\maxcode(\pi)(i)=\#\{j<i:\pi(j)>\pi(i)\}.\]
In other words, $\maxcode(\pi)(i)$ counts the number of crosses in the $i$th row of the Rothe diagram of $\pi$.
\end{defn}

\begin{ex}
    Let $\pi=13574862$. Then $\code(\pi)=(0,1,2,3,1,2,1,0,0,\cdots)$ and
        $\maxcode(\pi)=(0,0,0,0,2,0,2,6,0,0,\cdots)$.
\end{ex}

The following lemma is straightforward from the definition of $\code$ and $\maxcode$. We omit the proof.
\begin{lemma}
   \label{lem:code-maxcode-properties}
   Let $\pi$ be a permutation and suppose $\pi t_{\alpha \beta}\gtrdot \pi$. Let \[m_\alpha:=\#\{i>\beta: \pi(\alpha)<\pi(i)<\pi(\beta)\}\] and \[m_{\beta}:=\#\{i<\alpha: \pi(\alpha)<\pi(i)<\pi(\beta)\}.\] Then
   \begin{enumerate}[(a)]
       \item For $i\neq \alpha$ and $i\neq \beta$, \[\code(\pi)(i)=\code(\pi t_{\alpha\beta})(i)\] and \[\maxcode(\pi)(i)=\maxcode(\pi t_{\alpha\beta})(i).\] 
       \item \[\code(\pi t_{\alpha\beta})(\alpha)=\code(\pi)(\alpha)+m_{\alpha}+1\] and \[\code(\pi t_{\alpha\beta})(\beta)=\code(\pi)(\beta)-m_\alpha.\]
       \item \[\maxcode(\pi t_{\alpha\beta})(\alpha)=\maxcode(\pi)(\alpha)-m_\beta\] and \[\maxcode(\pi t_{\alpha\beta})(\beta)=\maxcode(\pi)(\beta)+m_\beta+1. \]
   \end{enumerate} 
\end{lemma}

From now on, for a word $\mathbf{a}$, we let $S(\mathbf{a})$ denote the insertion semistandard tableau obtained from performing Schensted's insertion. Given a semistandard tableau $S$, we let $\shape(S)$ denote its shape and $\colread(S)$ denote its column reading word; that is, the word obtained by reading $S$ by columns from left to right, and within each column from bottom to top.

\begin{lemma}
    \label{lem:left-absorb}
    Let $Q=\binom{\mathbf{a}}{k}$ be a $k$-biword, then the following are equivalent:
    \begin{enumerate}[(a)]
        \item $\binom{a}{k+1}Q\sim Q'$ where $Q'$ is a $k$-biword,
        \item left inserting $a$ into $\ins(\mathbf{a})$ does not add a new row, 
        \item $\binom{a}{k+1}\rightarrow \varphi(Q)$ does not create a descent at $k+1$.
    \end{enumerate}
   
\end{lemma}
\begin{proof}
    (b) $\Rightarrow$ (a): Since $\mathbf{a}\sim \colread(\ins(\mathbf{a}))$ we can without loss of generality assume $\mathbf{a}=\colread(\ins(\mathbf{a}))$. Let $c$ denote the bottom-most entry in the first column of $\ins(\mathbf{a})$. Suppose left inserting $a$ into $\ins(\mathbf{a})$ does not add a new row, then it must be the case that $a\le c$. Since $\binom{c}{k}$ is the first biletter of $Q$, we can apply (3) of the generalized Knuth relations and the result follows.

    (a) $\Rightarrow$ (c): This is straightforward since $\sim$ preserves the insertion BPD.

    (c) $\Rightarrow$ (b):
    Suppose left inserting $a$ into $\ins(\mathbf{a})$ does add a new row. Since a bumping step of Schensted left insertion on SSYTs corresponds to drooping into a blank tile of our left insertion on Grassmannian BPDs,   the left insertion of $\binom{a}{k}$ into $D:=\varphi(Q)$ starts from the leftmost $\rt$ in row $a>c$, never droops into a blank tile, and the last $\mindroop$
    must create a $\bt$ with the only $\rt$ in row $k+1$, so inserting $\binom{a}{k+1}$ must create a descent at $k+1$. See the illustration below.
    \begin{center}
        \includegraphics[scale=0.6]{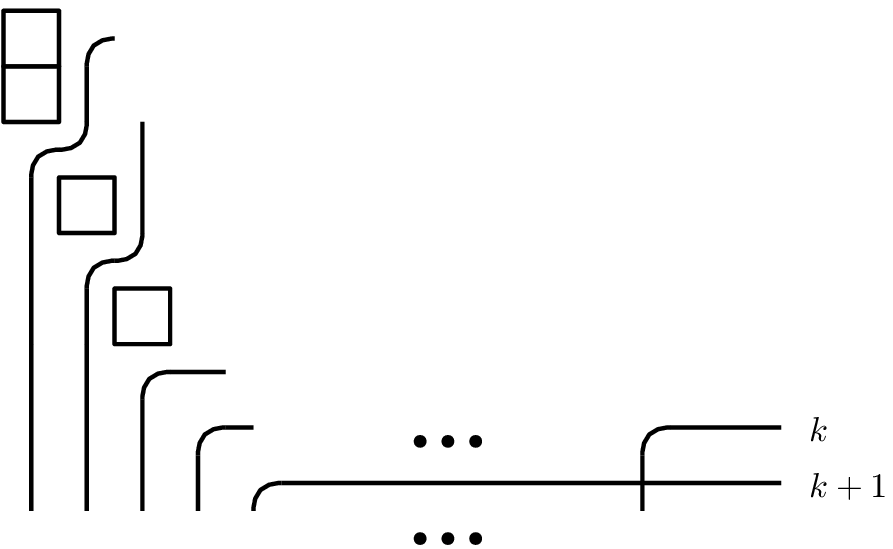}
    \end{center}
\end{proof}

\begin{lemma}
    \label{lem:spit-out-left}
    Let $Q=\binom{\mathbf{a}}{k}$ be a $k$-biword. Suppose $\lambda:=\shape(\ins(\mathbf{a}))$  where $\lambda^t=(c_1,\cdots, c_m)$ with $c_1\ge \cdots \ge c_m$. Namely, the column heights of $\lambda$ are $c_1,\cdots, c_m$. Then $Q\sim Q_1C$ where $Q_1=\binom{\mathbf{a}_1}{k+1}$ is a $(k+1)$-biword with $\shape(\ins(\mathbf{a}_1))^t=(c_2,\cdots, c_m)$ and $C=\binom{\mathbf{a}_c}{k}$ is a $k$-biword with $\shape(\ins(\mathbf{a}_c))^t=(c_1)$.
\end{lemma}
\begin{proof}
    Let $T$ be a standard tableau of shape $\lambda$ where the first column has content $1,\cdots, c_1$. Perform reverse left insertion on $\ins(\mathbf{a})$ using the order determined by $T$, we get $\mathbf{a}_1\mathbf{a}_c\sim\mathbf{a}$ where $\ell(\mathbf{a}_c)=c_1$ and $\shape(\ins(\mathbf{a}_c))^t=(c_1)$. By Lemma~\ref{lem:left-absorb}, $Q\sim \binom{\mathbf{a}_1}{k+1}\binom{\mathbf{a}_c}{k}$. Now $\ins(\mathbf{a}_1)\leftarrow \mathbf{a}_c=\ins(\mathbf{a})$. Since right inserting a column must add a column strip to the partition shape, we must have $\shape(\ins(\mathbf{a}_1))^t=(c_2,\cdots, c_m)$.
\end{proof}
\begin{ex} We use the notation $\frac{T}{k}$ where $T$ is an SSYT to denote the set of biwords $\binom{\mathbf{a}}{k}$ where $\ins(\mathbf{a})=T$. The following equivalences give examples for Lemma~\ref{lem:spit-out-left}:
\[\frac{\left(\begin{smallmatrix}
    1 & 1 & 2 & 3 \\
    2 & 3 & 4 &   \\
    3 &   &   &
\end{smallmatrix}\right)}{4} \sim \frac{\left(\begin{smallmatrix}
    1 & 2 & 2\\
    3 & 3 &  \\
\end{smallmatrix}\right)}{5}
\frac{\left(\begin{smallmatrix}
    1\\3\\4
\end{smallmatrix}\right)}{4}\sim
\frac{\left(\begin{smallmatrix}
    1 & 2 \\
    3 &
\end{smallmatrix}\right)}{6}\frac{\left(\begin{smallmatrix}
    2 \\
    3 \\
\end{smallmatrix}\right)}{5}
\frac{\left(\begin{smallmatrix}
    1\\3\\4
\end{smallmatrix}\right)}{4}\sim
\frac{(\begin{smallmatrix}
    1
\end{smallmatrix})}{7}
\frac{\left(\begin{smallmatrix}
    2 \\
    3 
\end{smallmatrix}\right)}{6}\frac{\left(\begin{smallmatrix}
    2 \\
    3 \\
\end{smallmatrix}\right)}{5}
\frac{\left(\begin{smallmatrix}
    1\\3\\4
\end{smallmatrix}\right)}{4}.
\]
\end{ex}
\begin{lemma}
    \label{lem:right-absorb}
    Let $Q=\binom{\mathbf{a}}{k+1}$ be a $k+1$-biword, then the following are equivalent:
    \begin{enumerate}[(a)]
        \item $Q\binom{a}{k}\sim Q'$ where $Q'$ is a $k+1$-biword,
        \item right inserting $a$ into $\ins(\mathbf{a})$ does not add a new column,
        \item $\varphi(Q)\leftarrow \binom{a}{k}$ does not create a descent at $k$.
    \end{enumerate}
\end{lemma}
\begin{proof}
This is similar to the proof of Lemma~\ref{lem:left-absorb}.    
\end{proof}

\begin{lemma}
    \label{lem:spit-out-right}
    Let $Q=\binom{\mathbf{a}}{k+1}$ be a $k+1$-biword. Suppose $\lambda:=\shape(\ins(\mathbf{a}))$  where $\lambda=(r_1,\cdots, r_m)$ with $r_1\ge \cdots \ge r_m$. Namely, the row lengths of $\lambda$ are $r_1,\cdots, r_m$. Then $Q\sim RQ_1$ where $Q_1=\binom{\mathbf{a}_1}{k}$ is a $k$-biword with $\shape(\ins(\mathbf{a}_1))=(r_2,\cdots, r_m)$ and $R=\binom{\mathbf{a}_r}{k+1}$ is a $k+1$-biword with $\shape(\ins(\mathbf{a}_r))=(r_1)$.
\end{lemma}
\begin{proof}
    This is similar to the proof of Lemma~$\ref{lem:spit-out-left}$.
\end{proof}
\begin{ex}
The following equivalences give examples for Lemma~\ref{lem:spit-out-right}:
\[\frac{\left(\begin{smallmatrix}
    1 & 1 & 2 & 3 \\
    2 & 3 & 4 &   \\
    3 &   &   &
\end{smallmatrix}\right)}{4} \sim 
\frac{\left(\begin{smallmatrix}
    1 & 3 & 3 & 4 
\end{smallmatrix}\right)}{4} \frac{\left(\begin{smallmatrix}
    1&2&3\\
    2& &
\end{smallmatrix}\right)}{3} \sim
\frac{\left(\begin{smallmatrix}
    1 & 3 & 3 & 4 
\end{smallmatrix}\right)}{4} \frac{\left(\begin{smallmatrix}
    2&2&3
\end{smallmatrix}\right)}{3}
\frac{\left(\begin{smallmatrix}
    1
\end{smallmatrix}\right)}{2}.
\]
\end{ex}

\begin{lemma}
    \label{lem:shortest-k-unique}
    Let $D\in\BPD(\pi)$, $k:=d_1(\pi)$ (namely $k$ is the smallest with $\maxcode(\pi)(k+1)\neq 0$), and $Q$ a biword with $\varphi(Q)=D$. If $Q=Q'\binom{\mathbf{a}}{k}$ where each $\binom{b_j}{k_j}$ in $Q'$ has $k_j>k$ and $\ell(\mathbf{a})=\maxcode(\pi)(k+1)$, then the $k$-biword $\binom{\mathbf{a}}{k}$ and the BPD $\varphi(Q')$ are unique.
\end{lemma}

\begin{proof}
    We successively perform right insertion of $Q$ and keep track of the maxcode of the permutations. Let $Q_i$ denote the prefix of $Q$ of length $i$ and let $\pi_i:=\perm(\varphi(Q_i))$. Let $m:=\ell(Q')$. For each $0\le i < m$, $\pi_{i+1}=\pi_i t_{\alpha_i\beta_i}$ for some $\beta_i > k+1$. Therefore $\maxcode(\pi_m)(j)=0$ for any $j\le k+1$. For each $m\le i <\ell(\pi)$, $\pi_{i+1}=\pi_i t_{\alpha_i \beta_i}$ for some $\beta_i>k$. Therefore $\maxcode(\pi_i)(j)=0$ for all $m\le i \le \ell(\pi)$ and $j < k+1$. Therefore to achieve $\ell(\mathbf{a})=\maxcode(\pi)(k+1)$, it must be the case that for $m\le i < \ell(\pi)$, $\maxcode(\pi_i)$ and $\maxcode(\pi_{i+1})$ agrees everywhere except $\maxcode(\pi_{i+1})(k+1)=\maxcode(\pi_i)(k+1)+1$. It follows that for $m\le i<\ell(\pi)$, $\pi_{i}=\pi_{i+1}t_{\alpha_i,k+1}$ where $\alpha_i\le k$ is the smallest position such that $\pi_{i+1}(\alpha_i)>\pi_{i+1}(k+1)$ by Lemma~\ref{lem:code-maxcode-properties}.  This determines the recording chain from $\pi_m$ to $\pi$, so $\binom{\mathbf{a}}{k}$ and $\varphi(Q')$ are unique.
\end{proof}

\begin{lemma}
    \label{lem:kmax-unique}
    Let $D\in\BPD(\pi)$, $k$ largest with $\maxcode(\pi)(k+1)\neq 0$, and $Q$ a biword with $\varphi(Q)=D$. If $Q=\binom{\mathbf{a}}{k}Q'$ where each $\binom{b_j}{k_j}$ in $Q'$ has $k_j<k$ and $\ell(\mathbf{a})=\maxcode(\pi)(k+1)$, then the $k$-biword $\binom{\mathbf{a}}{k}$ and the BPD $\varphi(Q')$ are unique.
\end{lemma}
\begin{proof}
    We successively perform left insertion of $Q$ and keep track of the maxcode of the permutations. 
    Let $Q_i$ denote the suffix of $Q$ of length $i$ and let $\pi_i:=\perm(\varphi(Q_i))$. Let $m:=\ell(Q')$. For each $0\le i <\ell(\mathbf{a})$, we have $\pi_{m+i+1}=\pi_{m+i}t_{\alpha_{m+i}\beta_{m+i}}$ where $\beta_{m+i}>k$. Namely, the insertion could only increase the maxcode of the permutation at positions larger than $k$. Since each $\binom{b_j}{k_j}$ in $Q'$ has $k_j<k$, the insertion of each biletter in $Q'$ cannot decrease the maxcode of the permutation at positions larger than $k$. Therefore to achieve $\ell(\mathbf{a})=\maxcode(\pi)(k+1))$ it must be the case that for $0\le i\le m$, $\maxcode(\pi_i)(k+1)=0$ and for $m\le i < \ell(\pi)$, $\maxcode(\pi_i)$ and $\maxcode(\pi_{i+1})$ agree everywhere except $\maxcode(\pi_{i+1})(k+1)=\maxcode(\pi_{i})(k+1)+1$. It follows that for $m\le i<\ell(\pi)$, $\pi_{i}=\pi_{i+1}t_{\alpha_i,k+1}$ where $\alpha_i\le k$ is the smallest position such that $\pi_{i+1}t_{\alpha_i,k+1}\lessdot \pi_{i+1}$ by Lemma~\ref{lem:code-maxcode-properties}. This determines the recording chain from $\pi_m$ to $\pi$ so $\binom{\mathbf{a}}{k}$ and $\varphi(Q')$ are unique.
\end{proof}

\begin{prop}
For any $D\in\BPD(\pi)$ there exists a unique plactic biword $Q_D^{\max}=\binom{a_1,\cdots, a_\ell}{k_1,\cdots, k_\ell}$ such that $\varphi(Q^{\max}_D)=D$ and the number of occurrence of $i$ in $(k_1,\cdots,k_\ell)$ is $\maxcode(\pi)(i+1)$. Define $Q^{\max}_D$ to be the \textbf{maxword} for $D$.
\end{prop}
\begin{proof}
This follows by induction using Lemma~\ref{lem:shortest-k-unique}.
\end{proof}

\begin{ex}
\label{ex:maxword}
    Let $\pi=13574862$ and $D\in\BPD(\pi)$ be the BPD below. 
    \begin{center}
        \includegraphics[scale=0.7]{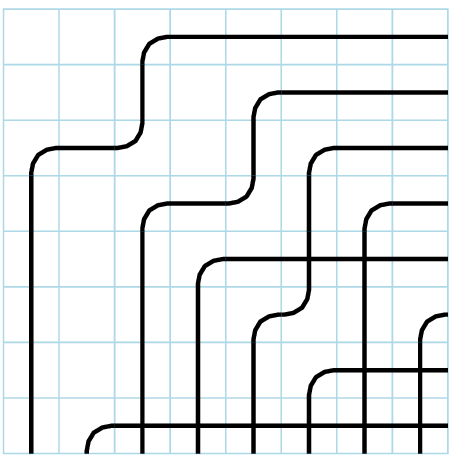}
    \end{center}
    Then \[Q_D^{\max}=\left(\begin{smallmatrix}
        7&6&5&4&2&1&2&1&3&2 \\
        7&7&7&7&7&7&6&6&4&4
    \end{smallmatrix}\right).\] 
Its recording chain for right insertion is $(12345678\lessdot_7 12345687 \lessdot_7 12345786\lessdot_7 12346875 \lessdot_7 12356784 \lessdot_7 12456783 \lessdot_7 13456782 \lessdot_6 13456872 \lessdot_6 13457862 \lessdot_4 13475862 \lessdot_4 13574862).$
\end{ex}

\begin{lemma}
\label{lem:grass-absorb}
Suppose $D\in\BPD(\pi)$ where $\pi$ is a Grassmannian permutation with descent at $k$, and let $Q=\binom{\mathbf{a}}{k}$ be a $k$-biword for $D$. Then if $\pi':=\perm(D\leftarrow \binom{a}{m})$ with $m<k$ stays a Grassmannian permutation, $Q\binom{a}{m}\sim Q'$ for any $k$-biword $Q'$ of $D\leftarrow\binom{a}{m}$.
\end{lemma}
\begin{proof}
We proceed by induction on $k-m$. The case when $m=k-1$ follows from Lemma~\ref{lem:right-absorb}. Now suppose $k-m>1$. Let $\lambda:=\shape(\ins(\mathbf{a}))$. By Lemma~\ref{lem:spit-out-right}, $Q\sim RQ_1$ where $R=\binom{\mathbf{a}_r}{k}$ is a length-$r_1$ $k$-biword, $\mathbf{a}_r$ weakly increasing, and $Q_1$ is a $(k-1)$-biword. Notice that $\code(\pi)(k)=r_1$, by the correspondence between partitions and Grassmannian permutations. 

We now argue that $\rho:=\perm(\varphi(Q_1)\leftarrow \binom{a}{m})$ must only have a descent at $k-1$. Since $\pi'$ is $k$-Grassmannian, $\pi'=\pi t_{\alpha\beta}$ where $\alpha\le m < k <\beta$. Therefore $\code(\pi')$ and $\code(\pi)$ agree everywhere except at $\alpha$ where $\code(\pi')(\alpha)=\code(\pi)(\alpha)+1$. In particular, $\code(\pi')(k)=r_1$. Consider left inserting $RQ_1\binom{a}{m}$ and keep track of the code of the permutation after the insertion of each biletter. The only times that the code at $k$ could increase are during the insertion of $R$, and since the permutations do not have descents $>k$ inserting each biletter in $R$ increases the code at $k$ by 1. Therefore, $\code(\rho)$ and $\code(\pi')$ agree at positions $\le k-1$, meaning that $\rho$ only has a descent at $k-1$. By induction hypothesis, $Q_1\binom{a}{m}\sim Q_1\binom{a}{k-1}$, since $\varphi(Q_1\binom{a}{m})=\varphi(Q_1\binom{a}{k-1})$. Therefore $Q\binom{a}{m}\sim RQ_1\binom{a}{m}\sim RQ_1\binom{a}{k-1}\sim Q\binom{a}{k-1}\sim Q\binom{a}{k}$ and the last step follows from Lemma~\ref{lem:right-absorb}.
\end{proof}

\begin{ex}
The following example illustrates Lemma~\ref{lem:grass-absorb}, with $Q=\binom{\mathbf{a}}{4}$ where $\mathbf{a}$ is any word in the Knuth class of the word $32341123$, $D=\varphi(Q)$, and $\binom{a}{m}=\binom{2}{2}$.
  \[\frac{\left(\begin{smallmatrix}
    1 & 1 & 2 & 3 \\
    2 & 3 & 4 &   \\
    3 &   &   &
\end{smallmatrix}\right)}{4} \frac{\left(\begin{smallmatrix}
    2
\end{smallmatrix}\right)}{2} \sim
\frac{\left(\begin{smallmatrix}
    1 & 3 & 3 & 4 
\end{smallmatrix}\right)}{4} \frac{\left(\begin{smallmatrix}
    1&2&3\\
    2& &
\end{smallmatrix}\right)}{3}
\frac{\left(\begin{smallmatrix}
    2
\end{smallmatrix}\right)}{2}\sim 
\frac{\left(\begin{smallmatrix}
    1 & 3 & 3 & 4 
\end{smallmatrix}\right)}{4} \frac{\left(\begin{smallmatrix}
    1&2&2\\
    2&3&
\end{smallmatrix}\right)}{3} \sim
\frac{\left(\begin{smallmatrix}
    1 & 1 & 2 & 2 \\
    2 & 3 & 3 &   \\
    3 & 4 &   &
\end{smallmatrix}\right)}{4}.
\]  
\end{ex}

\begin{defn}
    Let $\pi$ be a permutation where $d_1(\pi)=k$. Define $h(\pi)$ to be the permutation with $\code(h(\pi))(i)=\code(\pi(i))$ for $i\le k$  and $\code(h(\pi))=0$ for $i>k$. 
\end{defn}

\begin{ex}
    Let $\pi=13574862$, then $h(\pi)=13572468$.
\end{ex}


\begin{lemma}
    \label{lem:most-k-grass}
    Suppose $D\in\BPD(\pi)$, $\varphi(Q_1Q_2)=D$ where $Q_2$ is a $k$-biword with $\ell(Q_2)=\ell(h(\pi))$. Then each biletter $\binom{a_j}{k_j}$ in $Q_1$ has $k_j>k$ and $\varphi(Q_2)\in\BPD(h(\pi))$.
\end{lemma}
\begin{proof}
    Consider successively left inserting $Q_2$ into the identity BPD. Each step of insertion causes the permutation to change by right multiplication by $t_{\alpha\beta}$ for some $\alpha\le k <\beta$. Therefore, the code of the permutation strictly increases at $\alpha$ and otherwise remains the same.  Now since each biletter $\binom{a_j}{k_j}$ in $Q_1$ has $k_j\ge k$, when successively left inserting each biletter in $Q_1$, the code of the permutation for each position $i\le k$ is non-decreasing. Since $\ell(Q_2)=\ell(h(\pi))$ and $\code(h(\pi))$ agrees with $\code(\pi)$ for positions $i\le k$, the only possibility is for the code to increase by 1 at some $\alpha\le k$ when left inserting a biletter from $Q_2$ and remains the same for all $i\le k$ when inserting biletters from $Q_1$. The statements then follow.
\end{proof}

\begin{defn}
    Let $\pi$ be a permutation, and suppose $\rho<\pi$ is a permutation such that there exists $\alpha\ge 1$ where $\code(\pi)(i)=\code(\rho)(i)$ for $i<\alpha$, $\code(\rho)(\alpha)<\code(\pi)(\alpha)$, and $\code(\rho)(i)=0$ for all $i>\alpha$. Notice that $\alpha$ is unique when exists. Define $\delta(\rho):=\rho t_{\alpha\beta}$ where $\beta>\alpha$ is the smallest position with $\pi(\beta)>\pi(\alpha)$. 
\end{defn}

\begin{prop}
For any $D\in\BPD(\pi)$ there exists a unique plactic biword $Q_D^{\min}=\binom{a_\ell,\cdots, a_1}{k_\ell,\cdots, k_1}$ such that $\varphi(Q^{\min}_D)=D$ and the number of occurrence of $i$ in $(k_\ell,\cdots,k_1)$ is $\code(\pi)(i)$. Define $Q^{\min}_D$ to be the \textbf{minword} for $D$.
\end{prop}

\begin{proof}
    Let $\mathbf{ch}:=(\id=\pi_0\lessdot_{k_1}\cdots \lessdot_{k_\ell} \pi_\ell=\pi)$ where $\pi_i=\delta^i(\id)$ and $k_i=\alpha_i$ where $\pi_{i+1}=\pi_{i}t_{\alpha_i\beta_i}$ and $Q_D^{\min}:=\mathcal{L}^{-1}(D,\mathbf{ch})$. By definition of $\delta$, the number of occurrence of $i$ in $(k_\ell,\cdots, k_1)$ is $\code(\pi)(i)$. Uniqueness follows from the fact that the insertion of a biletter $\binom{a}{k}$ increases the code of the permutation at a position $\alpha\le k$. 
\end{proof}

\begin{ex}
    For $D\in\BPD(\pi)$ in Example~\ref{ex:maxword}, 
    \[Q^{\min}_D=\left(\begin{smallmatrix}
        7&2&6&5&1&2&4&1&3&2 \\
        7&6&6&5&4&4&4&3&3&2
    \end{smallmatrix}\right).\]
    Its recording chain for left insertion is 
    $12345678\lessdot_2 13245678 \lessdot_3 13425678 \lessdot_3 13524678 \lessdot_4 13542678 \lessdot_4 13562478 \lessdot_4 13572468 \lessdot_5 13574268 \lessdot_6 13574628 \lessdot_6 13574826 \lessdot_7 13574862$. 
\end{ex}

\section{Proof of the main theorem}

\label{sec:main-proof}
We prepare a few more technical lemmas that will be important for the main proof.
\begin{lemma}
    \label{lem:add-bottom-strip}
    Let $\mathbf{a}$ be a word and $\lambda:=\shape(\ins(\mathbf{a}))$ with $\lambda^t=(c_1,\cdots,c_m)$ where $c_1\ge \cdots \ge c_m$. Suppose $\mathbf{b}=b_l\cdots b_1$ for some $1\le l<m$ and let $\mathbf{b}_i=b_i\cdots b_1$ for each $1\le i\le l$. Let $Q_i:=\binom{\mathbf{b}_i}{k+1}$ and $Q:=\binom{\mathbf{a}}{k}$, and $\pi_i:=\perm(\varphi(Q_iQ))$. If for all $1\le i\le l$, $\code(\pi_i)(k+1)=i$, then $\shape(\ins(\mathbf{b}_i\mathbf{a}))^t=(c_1+1,\cdots, c_i+1, c_{i+1},\cdots, c_m)$ for all $1\le i\le l$.
\end{lemma}
\begin{proof}
    Suppose for all $1\le i\le l$, $\code(\pi_i)(k+1)=i$, namely successively left inserting each biletter in $Q_l$ into $\varphi(Q)$ increases the code of the underlying permutation at $k+1$ by 1.
    We first argue that $b_l\le\cdots \le b_1$. Suppose not. Then $Q_l\sim RQ_l'$ for $R$ a $(k+1)$-biword with length less than $l$ and $Q_l'$ a $k$-biword by Lemma~\ref{lem:spit-out-right}. Then $Q_lQ\sim RQ_l'Q$ where $R$ is a $(k+1)$-biword and $Q_l'Q$ a $k$-biword. 
    It is then impossible that $\code(\pi_l)(k+1)=l$.
    
    That $\shape(\ins(\mathbf{b}_1\mathbf{a}))^t=(c_1+1,c_2, \cdots, c_m)$ follows from Lemma~\ref{lem:left-absorb}. Now suppose there exists $j$ where $\shape(\ins(\mathbf{b}_j\mathbf{a}))^t=(c_1+1,\cdots, c_j+1, c_{j+1},\cdots, c_m)$ but inserting $b_{j+1}$ to $\ins(\mathbf{b}_j\mathbf{a})$ adds an entry to a column $p>j+1$ (adding an entry to columns 1 to $j$ is impossible since $b_l\le\cdots \le b_1$). Without loss of generality assume $\mathbf{a}=\mathbf{c}_1\cdots \mathbf{c}_m$ is a column reading word where each $\mathbf{c}_i$ corresponds to a column. By assumption, $\mathbf{b}_j\mathbf{c}_1\cdots \mathbf{c}_j\sim \mathbf{c}_1'\cdots \mathbf{c}_j'$ where $\ell(\mathbf{c}_i')=c_i+1$ for each $1\le i\le j$, and $\mathbf{c}_1'\cdots \mathbf{c}_j'$ is a column reading word. Therefore, by Lemma~\ref{lem:right-absorb}, $\binom{\mathbf{b}_j}{k+1}\binom{\mathbf{a}}{k}\sim \binom{\mathbf{c}_1'\cdots \mathbf{c}_j'}{k+1}\binom{\mathbf{c}_{j+1}\cdots \mathbf{c}_m}{k}$. Now $\binom{b_{j+1}\mathbf{c}_1'\cdots \mathbf{c}_j'}{k+1}\sim \binom{\mathbf{c}_1''\cdots \mathbf{c}_j''b_{j+1}'}{k+1}$, where $\mathbf{c}_1''\cdots \mathbf{c}_j''$ is the column reading word of the first $j$ columns after left inserting $b_{j+1}$ into $\ins(\mathbf{c}_1'\cdots \mathbf{c}_j')$ and $b_{j+1}'$ is the entry bumped out of $\mathbf{c}_j'$. By assumption on $b_{j+1}$ we must have $b_{j+1}'\le x$ where $x$ is the first letter of $\mathbf{c}_{j+1}$. This means     $\binom{b_{j+1}\mathbf{c}_1'\cdots \mathbf{c}_j'}{k+1}\binom{\mathbf{c}_{j+1}\cdots \mathbf{c}_m}{k}\sim \binom{\mathbf{c}_1''\cdots \mathbf{c}_j''}{k+1}\binom{b_{j+1}'\mathbf{c}_{j+1}\cdots \mathbf{c}_m}{k}$. By Lemma~\ref{lem:spit-out-right}, $\binom{\mathbf{c}_1''\cdots \mathbf{c}_j''}{k+1}\sim R_jQ_j$ where $R_j$ is a length-$j$ $(k+1)$-biword and $Q_j$ is a $k$-biword. In summary, $\binom{\mathbf{b}_{j+1}}{k+1}\binom{\mathbf{a}}{k}\sim R_jQ_j\binom{b_{j+1}'\mathbf{c}_{j+1}\cdots \mathbf{c}_m}{k}$, contradicting that $\code(\pi_{j+1})(k+1)=j+1$. 
\end{proof}

\begin{lemma}
    \label{lem:add-bottom-strip-after-insertion}
    Let $T\in \SSYT(\lambda)$ where $\lambda^t=(c_1,\cdots, c_m)$, $\mathbf{b}_i=b_i\cdots b_1$ with $i<m$ and $b_i\le \cdots \le b_1$,  $T':=T\leftarrow a$. If left inserting $\mathbf{b}_i$ into $T$ adds an entry in each of the first $i$ columns of $T$, then left inserting $\mathbf{b}_i$ adds an entry in each of the first $i$ columns in $T'$.
\end{lemma}
\begin{proof}
    If right inserting $a$ into $T$ adds an entry in a column $j>i$, then left insertion of $\mathbf{b}_i$ and right insertion of $a$ into $T$ affect independent columns and the result is immediate. 
    Otherwise, the commutativity of Schensted's left and right insertion allows us to compute the left insertion of $\mathbf{b}_i$ into $T'$ by first left inserting $\mathbf{b}_i$ into $T$ and then right inserting $a$. Let $T'':=\mathbf{b}_i\rightarrow(T\leftarrow a)=(\mathbf{b}_i\rightarrow T)\leftarrow a$. By assumption, left insertion of $\mathbf{b}_i$ adds a horizontal border strip to $T$ in the first $i$ columns, so in this case, $T$ is strictly contained in $\mathbf{b}_i\rightarrow T$. Now since right inserting $a$ adds an entry in a column $j\le i$, it must be the case that $T''$ differs from $T$ only in columns $\le i$ and $T''$ has $i+1$ more entries than $T$. Since insertion of $\mathbf{b}_i$ into an SSYT must always add a horizontal border strip, we know that the shape of $T''$ differs from that of $T'$ by a horizontal border strip. It follows that left insertion $\mathbf{b}_i$ into $T'$ must add an horizontal border strip in the first $i$ columns of $T'$. 
\end{proof}
We are ready to prove the main theorem.
\begin{proof}[Proof of Theorem~\ref{thm:main}]
We proceed by induction on $\ell:=\ell(\pi)$ and show that for any $D\in\BPD(\pi)$, $\binom{a}{k}$ a biletter with $k\le d_1(\pi)$ (recall that $d_1(\pi)$ denotes the first descent of $\pi$ and $d_2(\pi)$ denotes the last descent of $\pi$), $Q^{\max}_D\binom{a}{k}\sim Q^{\max}_{D'}$ where $D'=\varphi(Q_D^{\max}\binom{a}{k})$. We write $Q^{\max}_D=\binom{a_1,\cdots, a_\ell}{k_1,\cdots, k_\ell}$ and suppose $\perm(D')=\pi t_{\alpha\beta}$. Notice that by the definition of maxwords, $k_\ell=d_1(\pi)$. Furthermore, we argue that it must be the case that $\beta\le k_1+2$. Since $k_1$ is the largest index $i$ with $\maxcode(i+1)\neq 0$ (by definition of maxwords), $\pi(k_1+1+i)=k_1+1+i$ for all $i\ge 1$. Since $\pi t_{\alpha\beta}\gtrdot \pi$, we know that if $\beta > k_1+2$, it must be the case that $\alpha=\beta-1$. This contradicts  $\alpha\le k <\beta$ since $k\le d_1(\pi)$. 

\textbf{Case 1 ($k=k_\ell,\ \beta=k_1+1$ or $\beta=k_1+2$).} 

By Lemma~\ref{lem:code-maxcode-properties}, since $\beta>d_2(\pi)$, $\#\{i>\beta:\pi(\alpha)<\pi(i)<\pi(\beta)\}=0$, and we know that $\code(\pi t_{\alpha\beta})$ agrees with $\code(\pi)$ everywhere except $\code(\pi t_{\alpha\beta})(\alpha)=\code(\pi)(\alpha)+1$. Since $\alpha\le k$,  $\code(h(\pi t_{\alpha\beta}))$ agrees with $\code(h(\pi))$ everywhere except $\code(h(\pi t_{\alpha\beta}))(\alpha)=\code(h(\pi))(\alpha)+1$. Let $\lambda:=h(\pi)$ and $\lambda':=h(\pi t_{\alpha\beta})$. Let $\lambda_{\Par}$ and $\lambda'_{\Par}$ denote their corresponding partitions and suppose $(\lambda_{\Par})^t=(c_1,\cdots, c_m)$ and $(\lambda_{\Par}')^t=(c_1',\cdots, c'_{m'})$. Note that $m=\code(\lambda)(k)=\code(\pi)(k)$. 

Let $\mathbf{ch}=(\id\lessdot_{u_{\ell}} \pi_1\lessdot_{u_{\ell-1}}\cdots \lessdot_{u_{\ell+1-i}}\pi_i\lessdot_{\ell-i}\cdots \lessdot_{u_1}\pi_\ell=\pi$) be a chain that satisfies the following conditions:
\begin{enumerate}[(a)]
    \item $\pi_{\ell(\lambda)}=\lambda$ and $u_\ell=\cdots =u_{\ell-\ell(\lambda)+1}=k$
    \item $u_{\ell-\ell(\lambda)+1}\le \cdots \le u_1$
    \item If $l:=\code(\pi)(k+1)>0$, then for all $1\le j\le l$, $\pi_{\ell(\pi)+j}=\delta^j(\lambda)$ and $u_{\ell(\lambda)+j}=k+1$. (Note that $l$ and $\ell$ are different notations in this proof.)
\end{enumerate}
It is easy to see that $\mathbf{ch}$ always exists. Let $Q:=\mathcal{L}^{-1}(D,\mathbf{ch})$. By construction of $\mathbf{ch}$,  $Q=Q_0\binom{\mathbf{b}}{k+1}\binom{\mathbf{a}}{k}$ where each biword $\binom{v}{u}$ in $Q_0$ has $u>k+1$, and $\mathbf{b}$ is empty if $l=0$. Write $\mathbf{b}=b_l\cdots b_1$ when $l>0$. 

Let $\mathbf{a}':=\mathbf{a}a$. By Lemma~\ref{lem:most-k-grass}, $\varphi(\binom{\mathbf{a}'}{k})\in\BPD(\lambda')$. 
Note that $\shape(S(\mathbf{a}'))=\lambda'_{\Par}$.
Let $\mathbf{b}_i:=b_i\cdots b_1$ for each $1\le i\le l$ and $Q_i:=\binom{\mathbf{b}_i}{k+1}$. By construction of $\mathbf{ch}$, $\code(\perm(\varphi(Q_i\binom{\mathbf{a}}{k})))(k+1)=i$. Since $k$ is a descent of $\pi$, $\code(\pi)(k)>\code(\pi)(k+1)$ and thus $l<m$. By Lemma~\ref{lem:add-bottom-strip}, $\shape(\ins(\mathbf{b}_i\mathbf{a}))^t=(c_1+1,\cdots c_i+1, c_{i+1},\cdots, c_m)$ for all $1\le i \le l$. 
By Lemma~\ref{lem:add-bottom-strip-after-insertion}, $\shape(\ins(\mathbf{b}_i\mathbf{a}'))^t=(c'_1+1,\cdots c'_i+1, c'_{i+1},\cdots, c'_{m'})$ for all $1\le i \le l$.
Let $\mathbf{c}'_1\cdots\mathbf{c}'_{m'}\sim\mathbf{a}'$ be the column reading word of $\ins(\mathbf{a}')$. Then by Lemma~\ref{lem:right-absorb}, $\binom{\mathbf{b}}{k+1}\binom{\mathbf{a}'}{k}\sim \binom{\mathbf{b}\mathbf{c}'_1\cdots\mathbf{c}'_{l}}{k+1}\binom{\mathbf{c}'_{l+1}\cdots\mathbf{c}'_{m'}}{k}$. Furthermore by Lemma~\ref{lem:left-absorb}, $\binom{\mathbf{b}\mathbf{c}'_1\cdots\mathbf{c}'_{l}}{k+1}\binom{\mathbf{c}'_{l+1}\cdots\mathbf{c}'_{m'}}{k}\sim Q_c\binom{e_1,\cdots, e_r}{k}$ where $Q_c$ is a $k+1$-biword, $e_1>\cdots>e_{r}$ where $r=c'_{l+1}$. We now argue that $c'_{l+1}=\maxcode(\pi t_{\alpha\beta})(k+1)$. Consider the Rothe BPD of $\pi t_{\alpha\beta}$. The first $k$ rows of the Rothe BPD of $\pi t_{\alpha\beta}$ is identical the first $k$ rows of the Rothe BPD of $\lambda'$. The heights of columns of blank tiles the first $k$ rows are $c_1',\cdots, c'_{m'}$ by the correspondence between $\lambda'$ and $\lambda'_{\Par}$. Since $\code(\pi t_{\alpha\beta})(k+1)=l$, the column of blank tiles in column $\pi t_{\alpha\beta}(k+1)$ of the first $k$ rows has height $c'_{l+1}$. This implies that a pipe exits from a row in range $[k-c'_{l+1}+1,k]$ if and only if it crosses the pipe $\pi t_{\alpha\beta}(k+1)$. It follows that $c'_{l+1}=\maxcode(\pi t_{\alpha\beta})(k+1)$.

So far $Q^{\max}_D\binom{a}{k}\sim Q\binom{a}{k}$ by induction, and $Q\binom{a}{k}=Q_0\binom{\mathbf{b}}{k+1}\binom{\mathbf{a}'}{k}\sim Q_0Q_c\binom{e_1,\cdots, e_r}{k}$. By Lemma~\ref{lem:shortest-k-unique}, we must have $Q_{D'}^{\max}= Q_{D''}^{\max}\binom{e_1,\cdots,e_r}{k}$ where $\varphi(Q_{D''}^{\max})=\varphi(Q_0Q_c)$. Then $Q_D^{\max}\binom{a}{k}\sim Q_{D'}^{\max}$ by induction. 

\textbf{Case 2 ($k=k_\ell,\ \beta<k_1+1$).}
Since $\beta<k_1+1$, $m:=\maxcode(\pi t_{\alpha\beta})(k_1+1)=\maxcode(\pi)(k_1+1)$. Therefore by Lemma~\ref{lem:kmax-unique}, $Q_{D'}^{\max}=\binom{\mathbf{a}}{k_1}Q_1$ and $Q_D^{\max}\binom{a}{k}=\binom{\mathbf{a}}{k_1}Q_2$ such that $\ell(\mathbf{a})=m$, each biletter $\binom{b_j}{k_j}$ in $Q_1$ and $Q_2$ satisfies $k_j<k_1$, and $\varphi(Q_1)=\varphi(Q_2)$. The rest follows since $Q_1\sim Q_2$ by induction.

\textbf{Case 3 ($k<k_\ell,\ \beta>k_\ell$).} In this case $\pi t_{\alpha\beta}$ has the same descent set as $\pi$. 
Let $Q$ be a biword for $D$ such that $Q=Q_0\binom{\mathbf{a}}{k_\ell}$ where $\ell(\mathbf{a})=\ell(h(\pi))$. By Lemma~\ref{lem:most-k-grass}, $\varphi(\binom{\mathbf{a}}{k_\ell})\in\BPD(h(\pi))$. Consider the right insertion of $Q\binom{a}{k}$. Since inserting a bileter after $Q_0$ may only increase the code of the permutation at positions $\le k_\ell$ and $\ell(h(\pi t_{\alpha\beta}))=\ell(\mathbf{a})+1$, it is only possible for each insertion of a biletter in $Q_0$ to increase the code of the permutation at positions $>k$. Therefore, $\sum_{i>k}\code(\pi t_{\alpha\beta})(i)=\ell(Q_0)$.
Now consider the left insertion of $Q\binom{a}{k}$. Since $\code(\perm(\varphi(\binom{\mathbf{a}}{k_\ell}\binom{a}{k})))(i)=0$ for all $i>k $ and $\code(\perm(\varphi(\binom{\mathbf{a}}{k_\ell}\binom{a}{k})))(i)=\code(\pi t_{\alpha\beta})(i)$ for all $i\le k$, the insertion of letters in $Q_0$ must only increase the code at positions $>k$ and does not affect positions $\le k$.  
It follows that $\perm(\varphi(\binom{\mathbf{a}}{k_\ell}\binom{a}{k}))$ must be $k_\ell$-Grassmannian. Therefore $\binom{\mathbf{a}}{k_\ell}\binom{a}{k}\sim \binom{\mathbf{a}}{k_\ell}\binom{a}{k_\ell}$ by Lemma~\ref{lem:grass-absorb} and the rest follows by the first two cases.

\textbf{Case 4 ($k<k_\ell,\ \beta\le k_\ell$).}
In this case $\pi t_{\alpha\beta}$ gains a new descent at $k$ and $Q_D^{\max}\binom{a}{k}=Q_{D'}^{\max}$.
\end{proof}

\begin{ex}
We show an example demonstrating Case 1 of the proof of Thereom~\ref{thm:main} using $D\in\BPD(\pi)$ from Example~\ref{ex:maxword} and $\binom{a}{k}=\binom{1}{4}$. 

\begin{align}
    Q_D^{\max}\left(\begin{smallmatrix}
a\\k\end{smallmatrix}\right)&=\left(\begin{smallmatrix}
        7&6&5&4&2&1&2&1&3&2 \\
        7&7&7&7&7&7&6&6&4&4
\end{smallmatrix}\right)\left(\begin{smallmatrix}
1\\4\end{smallmatrix}\right)\\
&\sim \left(\begin{smallmatrix}
        7&6&5&4&2&1&2&1&3&2 \\
        7&6&6&6&6&6&6&6&4&4
\end{smallmatrix}\right)\left(\begin{smallmatrix}
1\\4\end{smallmatrix}\right)\\
&\sim \left(\begin{smallmatrix}
        7&2&6&5&4&2&1&1&3&2 \\
        7&6&6&5&5&5&5&5&4&4
\end{smallmatrix}\right)\left(\begin{smallmatrix}
1\\4\end{smallmatrix}\right) \\
&\sim \left(\begin{smallmatrix}
        7&2&6&5&4&2&1&3&1&2 \\
        7&6&6&5&5&5&5&4&4&4
\end{smallmatrix}\right)\left(\begin{smallmatrix}
1\\4\end{smallmatrix}\right) \\
&\sim \left(\begin{smallmatrix}
        7&2&6&5&4&2&1&3&1&2 \\
        7&6&6&5&4&4&4&4&4&4
\end{smallmatrix}\right)\left(\begin{smallmatrix}
1\\4\end{smallmatrix}\right) \\
&\sim \left(\begin{smallmatrix}
        7&2&6&5&4&3&2&1&2&1&1 \\
        7&6&6&5&4&4&4&4&4&4&4
\end{smallmatrix}\right) \\
&\sim \left(\begin{smallmatrix}
        7&2&6&5&4&3&2&1&2&1&1 \\
        7&6&6&5&4&4&4&4&4&4&4
\end{smallmatrix}\right) \\
&\sim \left(\begin{smallmatrix}
        7&2&6&5&4&3&2&1&2&1&1 \\
        7&6&6&5&5&5&5&5&4&4&4
\end{smallmatrix}\right) \\
&\sim \left(\begin{smallmatrix}
        7&2&6&5&4&3&2&1&2&1&1 \\
        7&6&6&5&5&5&5&5&4&4&4
\end{smallmatrix}\right)\\
&\sim \left(\begin{smallmatrix}
        7&2&6&5&4&3&2&1&1&2&1 \\
        7&6&6&5&5&5&5&5&5&4&4
\end{smallmatrix}\right) \\
&\sim \left(\begin{smallmatrix}
        7&6&5&4&3&2&1&2&1&2&1 \\
        7&6&6&6&6&6&6&6&6&4&4
\end{smallmatrix}\right) \\
&\sim \left(\begin{smallmatrix}
        7&6&5&4&3&2&1&2&1&2&1 \\
        7&7&7&7&7&7&7&6&6&4&4
\end{smallmatrix}\right) = Q^{\max}_{D'}
\end{align}
The equivalence from (1) to (5) corresponds to an application of the inductive hypothesis in the proof. The biword in step (5) corresponds to the biword $Q$ with left recording chain $\mathbf{ch}$, the biword in step (10) corresponds to $Q_0Q_c\binom{e_1,\cdots,e_r}{k}$. Finally, the equivalence from (10) to (12) corresponds to another application of the inductive hypothesis. The resulting BPD of the insertion is shown below.
\begin{center}
    \includegraphics[scale=0.7]{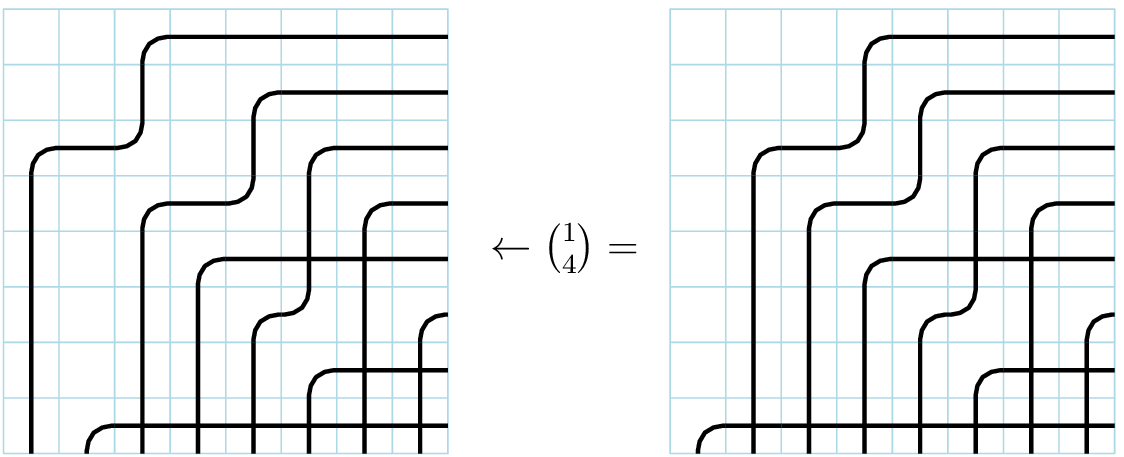}
\end{center}
\end{ex}

\bibliographystyle{alpha}
\bibliography{ref}
\end{document}